\documentclass[twocolumn]{article} 
\usepackage{algorithm, algorithmic, setspace}
\usepackage{graphicx} 
\usepackage{amsmath} 
\usepackage{amssymb}  
\usepackage{balance}
\usepackage{amsfonts}
\usepackage{color}
\usepackage{subfig}
\usepackage[english]{babel}
\usepackage[latin1]{inputenc}
\usepackage{siunitx}
\usepackage{lipsum}
\usepackage{cite}
\usepackage{hyperref}
\usepackage{verbatim}
\usepackage{tikz,pgfplots}

\newtheorem{remark}{Remark}
\newtheorem{definition}{Definition}

\newtheorem{theorem}{Theorem}
\newtheorem{lemma}{Lemma}

\newtheorem{assumption}{Assumption}

\newtheorem{corollary}{Corollary}

\newcommand{\sto}{ \rm{s.t.} }

\newcommand{\norm}[1]{\left\lVert #1 \right\rVert}

\newcommand{\xstar}{x^{\star}}

\newcommand{\argmin}[1]{\underset{#1}{\mathrm{argmin\,}}}

\newcommand{\R}{\mathbb{R}}

\newcommand{\prox}{\mathrm{prox}}
\newcommand{\sign}{\mathrm{sign}}
\newcommand{\kron}{\otimes}

\definecolor{green2}{rgb}{0.2, 0.75, 0.2}

\title{\LARGE \bf Discrete-time feedback linearization control\\for nonsmooth constrained optimization}

\author{V. Cerone \and S. M. Fosson\thanks{$^{*}$ Corresponding author. The work of S. Pirrera received funding from the European Union's Horizon Europe programme under the Marie Sk\l{}odowska-Curie grant agreement No.~101276493. The authors are with the Dipartimento di Automatica e Informatica, Politecnico di Torino, Corso Duca degli Abruzzi 24, 10129 Torino, Italy;  e-mail: \{vito.cerone,sophie.fosson,simone.pirrera,alice\_re, \\diego.regruto\}@polito.it.
}\and  S. Pirrera\and  A. Re \and  D. Regruto  
}
\date{}

\usetikzlibrary{plotmarks, fit}

\newcommand{\tikzref}[1][]{%
    \begin{tikzpicture}[baseline,yshift=0.3em]
        \draw [mark repeat=2,mark phase=2,#1]
        plot coordinates {
            (0cm,0cm)
            (0.25cm,0cm)
            (0.5cm,0cm)
        }; \node[fit=(current bounding box),inner xsep=0.1em]{};
    \end{tikzpicture}%
}

\begin{document}
\maketitle

\begin{abstract}
    We propose a discrete-time feedback linearization control approach to solve nonsmooth linearly constrained composite optimization problems. By using proximal operators, we construct a dynamical system whose equilibria correspond to the stationary points of the optimization problem. 
    Interpreting the Lagrange multipliers as control inputs, we employ feedback linearization control to steer the obtained dynamical system to a stable equilibrium. 
    We analyze the convergence of the resulting closed-loop system both in the strongly convex setting and under the Polyak--\L{}ojasiewicz condition. Finally, we illustrate the applicability of the approach through numerical experiments.
\end{abstract}

\section{Introduction}\label{sec:IN}
%
Nonsmooth constrained optimization problems play a central role in many areas, including machine learning, system identification, and control~\cite{has15book}; however, their solution is challenging due to the combined presence of nondifferentiability and constraints.
Nonsmooth terms arise in applications for different purposes. For instance, to promote sparsity~\cite{has15book}, low-rank solutions~\cite{jiang14}, and to handle constraints~\cite{cerone_alternating_2023}.

For nonsmooth optimization, proximal operator methods naturally extend gradient-based approaches that fail due to nondifferentiability~\cite{par13}. In this context, some well-established algorithms are the proximal gradient descent~\cite{par13} and the alternating direction method of multipliers (ADMM) \cite{boy10}. In particular, when the nonsmooth term is the $\ell_1$ norm, widely used approaches include the iterative shrinkage-thresholding algorithm~\cite{dau04} and linearized Bregman iterations~\cite{cai09}.
A possible approach to analyzing the convergence of iterative optimization algorithms is to model them as continuous-time or discrete-time dynamical systems and use control-theoretic tools. Building on this perspective, which dates back to~\cite{arr58}, recent works analyze convergence and design novel optimization algorithms using feedback control techniques.
For example,~\cite{Scherer23} employs integral quadratic constraints to develop accelerated gradient descent algorithms for unconstrained problems. The work~\cite{cmo24} proposes proportional-integral (PI) and feedback linearization (FL) methods for smooth equality-constrained problems, while~\cite{allibhoy24} employs control barrier functions to handle smooth problems with both equality and inequality constraints. These works interpret the Lagrange multipliers as control inputs of a fictitious plant; we refer to this framework as controlled multipliers optimization (CMO).
In the context of time-varying optimization,~\cite{CASTI2025112107} and~\cite{astolfi_repetitive_2025} address online constrained optimization exploiting the internal model principle and repetitive control tools, respectively.

The works~\cite{cent25} and~\cite{nscmo26} introduce and analyze CMO techniques for nonsmooth optimization in the continuous-time framework, based on PI control. However, the continuous-time setting requires numerical discretization for implementation, which may limit the computational efficiency, as optimal discretization strategies remain unexplored.

To address this limitation, this paper develops a novel discrete-time CMO-based algorithm for nonsmooth composite and constrained optimization. 
Motivated by the faster convergence rate of the FL approach compared to the PI one in the smooth continuous-time setting~\cite{cmo24}, here we adopt FL control instead of the PI-based schemes in~\cite{cent25,nscmo26}. To this end, we restrict our attention to linearly-constrained problems, for which a tractable discrete-time FL controller can be obtained. 
While many proximal methods handle equality constraints via indicator functions~\cite{dhi19}, this can be challenging when the associated proximal operators lack closed-form solutions or are computationally expensive~\cite{cent25}. Therefore, we explicitly retain the equality constraints to facilitate computational tractability. 

The theoretical contribution of this work is twofold. First, we analyze the algorithm's convergence and establish global exponential stability in the strongly convex setting via contractivity analysis. Second, we generalize the result to nonconvex problems by using the proximal Polyak-\L{}ojasiewicz (P\L{}) condition~\cite{karimi16} and Lyapunov arguments. 
Finally, we assess the practical performance of the proposed algorithm through three numerical experiments, including an equality-constrained Lasso problem and the training of a single-layer neural network with a nonsmooth activation function.

\textbf{Outline.}~Sec.~\ref{sec:PS} introduces the problem formulation and the required theoretical background. 
Sec.~\ref{sec:PA} develops the proposed algorithm and analyzes its convergence in both the strongly convex setting and under the P\L{} inequality. 
Sec.~\ref{sec:NR} illustrates the effectiveness of the proposed method through numerical examples. 
Finally, Sec.~\ref{sec:CON} concludes the paper.

\textbf{Notation.}~For a vector $x \in \R^n$, $\norm{x}_p = (\sum_{i=1}^n \vert x_i\vert^p)^{1/p}$ denotes the $\ell_p$ norm and $\norm{x}_Q = \sqrt{x^\top Q x}$, where $Q \in \R^{n,n}$, is the weighted $Q$ norm. Given a matrix $A \in \R^{n,m}$, $A^\dagger$ is its Moore-Penrose pseudoinverse and $\norm{A}_2$ is the induced $\ell_2$ norm. $\rho(A)$ denotes the spectral radius of $A$.
Given a discrete-time signal $x$, $x^+$ denotes the forward time shift operator applied to $x$. $\iota_\mathcal{K}(x)$ is the indicator function of a set $\mathcal{K}$, i.e., $\iota_\mathcal{K}(x)=0$ if $x \in \mathcal{K}$ and $\iota_\mathcal{K}(x)=+\infty$ otherwise. For a map $\phi: \mathcal{X} \rightarrow \mathcal{Y}$, $\mathrm{Im}(\phi) \subseteq \mathcal{Y}$ is the image of $\phi$. 

\section{Problem statement and Background}
\label{sec:PS}
Let $f(x): \R^n \to \R$ be a smooth function and $g(x): \R^n \to \R$ be possibly nonsmooth. We consider the equality-constrained optimization problem
\begin{equation}\label{eq:opt_problem}
\begin{aligned}
    \min_{x \in \R^n}& \, f(x) + g(x) \\ & {\sto} \quad h(x) = Cx+d =0
\end{aligned} 
\end{equation}
where $C\in\R^{m,n}$ and $d\in\R^m$.
Typically, $f(x)$ denotes a cost function to be minimized, while $g(x)$ is a regularization term that promotes structural properties of $x$ or an indicator function that encodes possible additional problem constraints. 
%
%
In the remainder of this paper, we will consider the following assumptions.
\begin{assumption}\label{assumption_g}
    $g(x)$ is proper, lower semi-continuous and convex.
\end{assumption}
\begin{assumption}\label{assumption_C}
    $C$ is of full row rank, i.e., $\mathrm{rank}(C) = m$.
\end{assumption}
Under Assumption~\ref{assumption_C}, the matrix $CC^\top$ is invertible. This is without loss of generality: if two rows of $C$ are linearly dependent, the corresponding constraints are either equivalent or mutually exclusive. In the first case, one row of $C$ and the corresponding entry of $d$ can be removed to restore linear independence; in the second case, the problem is infeasible. For further details, we refer the reader to~\cite{qu19,cmo24}.

The goal of this work is to develop a discrete-time optimization algorithm grounded in a control-theoretic framework.  
Before presenting and analyzing the proposed method, we first introduce the necessary theoretical background. 
\subsection{Proximal operators}
Proximal operators play a central role in algorithms for nonsmooth optimization.
Given $g(x): \mathbb{R}^{n} \mapsto \mathbb{R} \cup \{+ \infty\}$, the proximal operator $\text{prox}_{\mu g}(v) : \mathbb{R}^{n} \mapsto \mathbb{R}^n$ of the scaled function $\mu g(x)$, where $\mu > 0$, is defined as:
\begin{equation}
    \text{prox}_{\mu g} (v) = \arg\min_{x \in \mathbb{R}^{n}} \left( g(x) + \frac{1}{2 \mu} \|x-v\|_2^2 \right).
\end{equation}

Under Assumption~\ref{assumption_g}, the proximal operator is a single-valued firmly non-expansive map, i.e., for any $u, v \in  \mathbb{R}^{n}$:
\begin{align*}
    & \| \prox_{\mu g}(u) - \prox_{\mu g}(v)\|_2^2  \leq \\
    & \qquad \leq\langle u - v, \text{prox}_{\mu g}(u) - \text{prox}_{\mu g}(v) \rangle.
\end{align*}
The Moreau envelope of $g$ is defined as:
\begin{equation}
     M_{\mu g}(v) = \inf_{x \in \R^n} \left (g(x) + \frac{1}{2 \mu}   \| x-v \|^2_2 \right).
\end{equation}
$M_{\mu g}(v) $ has domain $\R^n$ and is continuously differentiable, even when $g(x)$ itself is not. 
The gradient of the Moreau envelope can be computed as
\begin{equation}\label{eq:grad_moreau_envelope}
    \nabla M_{\mu g}(v) = \frac{1}{\mu}(v - \prox_{\mu g}(v)).
\end{equation}

\subsection{Feedback control of the Lagrange multipliers}\label{sec:cmo}
The CMO framework is a continuous-time, control-theoretic approach for solving constrained optimization problems~\cite{allibhoy24,cmo24}. The core of this approach lies in associating the optimization problem with a dynamical system in which the constraints are interpreted as outputs and the Lagrange multipliers as control inputs.
Given the optimization problem 
\begin{equation} \label{eq:opt_smooth}\begin{aligned}
    &\min_{x \in \mathbb{R}^n}  \, f(x) &~~{\sto} \quad h(x) = 0,
\end{aligned} \end{equation}   
\cite{cmo24} defines the fictitious plant
\begin{equation}\label{eq:cmo_plant}
    \mathcal{P}_1: \left\{ \begin{split}
        &\dot x = -\nabla f(x) - J_h^\top(x)\lambda \\
        &y = h(x)
    \end{split}\right.
\end{equation}
where $\lambda \in \mathbb{R}^m$ are the Lagrange multipliers and $J_h(x)$ is the Jacobian of $h(x)$. The equilibria of~\eqref{eq:cmo_plant} correspond to the stationary points of Problem~\eqref{eq:opt_smooth}. As shown in~\cite[Lemma~1]{cmo24}, if a controller stabilizes the closed-loop system while driving $y$ to zero, the resulting trajectories converge to a stationary point of~\eqref{eq:opt_smooth}. 
In~\cite{cent25} and~\cite{nscmo26}, this framework is extended to problems of the form~\eqref{eq:opt_problem} to handle nonsmooth terms.
In particular,~\cite{nscmo26} uses the Moreau envelope $M_{\mu g}$ to define the proximal augmented Lagrangian
\begin{align}\label{eq:prox_aug_lagr_nscmo}
    \mathcal{L}_{\mu} (x,\alpha, \lambda) 
    \!=\!f(x)\! + \!M_{\mu g}(x \!+\! \mu \alpha) 
    \!- \!\frac{\mu}{2} \| \alpha \|_2^2 
    \!+\! \lambda^\top h(x).
\end{align}
Based on~\eqref{eq:prox_aug_lagr_nscmo}, the following fictitious plant is defined:
\begin{equation}\label{eq:plant_nscmo}
\mathcal{P}_2:\;
\begin{cases}
\dot{x} = - \nabla_x \mathcal{L}_{\mu} (x,\alpha, \lambda) \\
y_1 = x - \prox_{\mu g}\bigl(x + \mu \alpha\bigr) \\
y_2 = h\bigl(x\bigr).
\end{cases}
\end{equation}
According to \cite[Lemma 1]{nscmo26}, the equilibrium points of~\eqref{eq:plant_nscmo} are stationary points of Problem \eqref{eq:opt_problem} if and only if $y_1=0$ and $y_2 = 0$ at equilibrium.
In~\cite{nscmo26}, the multiplier $\lambda$ is updated using a PI controller:
\begin{equation}
    \dot \lambda(t) = k_p J_h(x)\dot x + k_i h(x),
\end{equation}
while two alternatives are proposed for $\alpha$: the static law $\alpha = - \nabla f(x)$ and a dynamic controller that extends the one defined by purely integral action. See \cite{nscmo26} for details.
\section{Proposed approach}\label{sec:PA}
We begin by recasting the CMO framework in a discrete-time setting. We build upon the dynamics resulting from the static controller $\alpha = -\nabla f(x)$ proposed in \cite{nscmo26}. This leads to the continuous-time dynamics
\begin{equation}\label{eq:static_prox_cmo}
    \dot{x} = - \frac{1}{\mu} x + \frac{1}{\mu} \prox_{\mu g}(x - \mu   \nabla f(x)) - C^\top \lambda.
\end{equation}
At this stage, we retain the freedom to design the controller for $\lambda$.
We apply forward Euler discretization to~\eqref{eq:static_prox_cmo} with step size $\mu$, obtaining the discrete-time fictitious plant
\begin{equation}\label{eq:plant}
    \mathcal{P}: \left\{ \begin{split}
        &x^+ = \prox_{\mu g}(x - \mu \nabla f(x)) - \mu C^\top \lambda\\
        &y = h(x).
    \end{split}\right.
\end{equation}
Thus, problem~\eqref{eq:opt_problem} is recast as a control problem in the CMO framework, as stated in the following corollary of~\cite[Lemma~1]{nscmo26}.
\begin{corollary}[Equilibria of $\mathcal{P}$]
    $x^\star$ is an equilibrium point of \eqref{eq:plant} satisfying $h(\xstar)=0$ if and only if $\xstar$ is a stationary point of the optimization problem in Eq.~\eqref{eq:opt_problem}. 
\end{corollary}

\textbf{Proof.}~System~\eqref{eq:plant} is obtained via forward Euler discretization of~\eqref{eq:static_prox_cmo} with stepsize $\mu$. Since forward Euler discretization preserves the set of equilibria for any nonzero stepsize (see, e.g., \cite{stuart1998dynamical}), and $\mu>0$ by construction, the result holds if the stationary points of the continuous-time dynamics \eqref{eq:static_prox_cmo} such that $h(\xstar)=0$ characterize the stationary points of \eqref{eq:opt_problem}. This is true by direct application of~\cite[Lemma~1]{nscmo26}. 

We now design a controller $\mathcal{K}$ that drives $\mathcal{P}$ toward an equilibrium satisfying $y=0$ and guarantees stability of $(x^\star,\lambda^\star)$ using FL. 
Let us define the continuous map
\begin{equation}\label{eq:Pi}
    \Pi(x) \doteq \prox_{\mu g}(x - \mu \nabla f(x)).
\end{equation}
With this definition, we can write the output's dynamics as
\begin{equation}\label{eq:y_plus}
    y^+ = C x^+ + d 
    = C\big( \Pi(x) - \mu C^\top \lambda \big) + d.
\end{equation}
We select the input $\lambda$ so as to impose $ y^+ = v$, obtaining
\begin{equation}\label{eq:lambda_FL}
    \lambda = \tfrac{1}{\mu}(CC^\top)^{-1}\big( C \Pi(x) + d - v \big),
\end{equation}
where $v$ is regulated by the linear dynamic controller
\begin{equation}\label{eq:v_dynamics}
    v^+ = K v + L y,
\end{equation}
and $K,L \in \mathbb{R}^{m \times m}$ are design matrices.
The controller dynamics $\mathcal{K}$ is therefore given by Eqs.~\eqref{eq:v_dynamics}--\eqref{eq:lambda_FL},
and the resulting closed-loop dynamics is:
\begin{subequations}\label{eq:FL-prox-CMO}
    \begin{align}
        & x^+ = \Pi(x) - C^\top (CC^\top)^{-1} (C \, \Pi(x) + d - v) \\
        & v^+ = K v + L (Cx + d).
    \end{align}
\end{subequations}
By construction, the $(y,v)$-subsystem evolves linearly as
\begin{equation}\label{eq:N}
    \begin{bmatrix} y^+ \\ v^+\end{bmatrix} = N\begin{bmatrix} y \\ v\end{bmatrix},
    \qquad
    N = \begin{bmatrix}
        0 & I \\
        L & K
    \end{bmatrix}.
\end{equation}
We design $K$ and $L$ so that $N$ is Schur stable with decay rate $r_e \in (0,1)$, i.e., there exists a symmetric matrix $Q = Q^\top \succ 0$ such that
\begin{equation}\label{eq:cond_N_stab}
    N Q N^\top - (1 - r_e) Q \prec 0
\end{equation}
We notice that it is possible to arbitrarily select the convergence rate $\rho(N) \leq \sqrt{1-r_e} <1$ by appropriately selecting $K$ and $L$, for instance by using semidefinite programming, as illustrated in~\cite{boyd_linear_1994}. 

\begin{remark}
    The proposed method is similar to the projected FL-CMO scheme in~\cite{pirrera2025}, which introduces a projection step after discretizing the FL-based algorithm in~\cite {cmo24}.
    Unlike~\cite{pirrera2025}, however, our approach explicitly decouples the constraint correction term and the proximal operator. This structural difference stems from the fact that our method is derived from the proximal augmented Lagrangian framework of~\cite{nscmo26}, whereas the approach in~\cite{pirrera2025} is based on the Lagrangian formulation introduced in~\cite{cmo24}. 
    Consequently, the constraint-related terms enter the equations differently, making our method easier to interpret and analyze. 
    Moreover, a key limitation of~\cite{pirrera2025} is that it only addresses the case of quadratic 
    $f(x)$. In contrast, our framework naturally extends to more general composite problems.
\end{remark}

\subsection{Convergence analysis: strongly convex case}\label{CA1}
In this section, we analyze the convergence of the dynamics~\eqref{eq:FL-prox-CMO}. Let us consider the following assumption.
\begin{assumption}\label{assumption_f}
    $f(x)$ is $m_f$-strongly convex, continuously differentiable with $L_f$-Lipschitz continuous gradient.
\end{assumption}

\noindent The following results hold.
\begin{lemma}\cite[Lemma~1]{qu19}\label{prop_B}
    Under Assumption~\ref{assumption_f}, for any $a,b\in\R^n$,  there exists a symmetric matrix $B(a,b)\in \R^{n,n}$ satisfying  $m_f I \preceq B(a,b) \preceq L_f I$ such that
    \begin{equation}\label{eq:B}
        \nabla f(a) - \nabla f(b)\!=\!B(a,b)\, (a-b).
    \end{equation}
\end{lemma}

\begin{lemma}\cite[Lemma 2]{nscmo26}\label{prop_D}
    Under Assumption~\ref{assumption_g}, for any $a,b\in\R^n$, there exists a symmetric matrix $D(a,b)\in \R^{n,n}$ satisfying $0 \preceq D(a,b) \preceq I $  such that
    \begin{equation}\label{eq:D}
         \prox_{\mu g}(a) -\prox_{\mu g}(b)\! =\!D(a,b)\, (a-b).
    \end{equation}
\end{lemma}
We refer the reader to~\cite{qu19,nscmo26} for additional details.
%
\begin{theorem}\label{th1}
    Let Assumptions~\ref{assumption_g}, \ref{assumption_C}, and \ref{assumption_f} hold. 
    Let the matrices $K, L$ be such that $N$ in Eq.~\eqref{eq:N} satisfies~\eqref{eq:cond_N_stab} for some $Q = Q^\top \succ 0$ and some $r_e \in (0,1)$. Then, for all $\mu \in \left(0,\frac{2}{L_f}\right)$, the algorithm defined by~\eqref{eq:FL-prox-CMO} admits a unique globally exponentially stable equilibrium point $(x^\star, v^\star)$.
\end{theorem}

\textbf{Proof.}~
    Let $C^\perp \in \mathbb{R}^{(n-m)\times n}$ have rows forming an orthonormal basis of the null space of $C$, so that $C^\perp C^\top = 0$ and $C^\perp C^{\perp\top} = I_{n-m}$. The coordinate transformation $z \doteq [\eta^\top,~y^\top,~v^\top]^\top$ with $\eta \doteq C^\perp x$ and $y = Cx + d$  is affine and invertible because $\begin{bmatrix} C \\ C^\perp \end{bmatrix}$ is nonsingular. 
    %
    Using~\eqref{eq:FL-prox-CMO}, the update in the new coordinates can be formulated as $z^+ = F(z)$, where
    \begin{equation}\label{eq:f_z}
        F(z) = 
        \begin{bmatrix}
             C^\perp \Pi(C^{\perp \top} \eta + C^\dagger (y-d)) \\
            v  \\
            K v + Ly 
        \end{bmatrix}.
    \end{equation}
    We now prove that $F$ is globally contractive with respect to a weighted $P$-norm, i.e., there exist $P = P^\top \succ 0 $ and $r \in (0,1)$ such that
    \begin{equation}\label{eq:contract}
        \| F(z_1) - F(z_2) \|_P \leq r \| z_1 - z_2 \|_P, \quad \forall z_1,z_2 \in \R^{n+m}.
    \end{equation}
    Exploiting Eqs.~\eqref{eq:B} and~\eqref{eq:D}, we obtain
    \begin{equation}\label{eq:matrice_M}
        F(z_1) \!-\! F(z_2)\! = \!M(z_1,z_2)\,(z_1 - z_2),
    \end{equation}
    where $M \equiv M(z_1,z_2) \doteq \begin{bmatrix} A(z_1,z_2)\ & \Phi(z_1,z_2)\ \\ 0 & N \end{bmatrix}$,
    \begin{equation*}
    \begin{split}
      &A(z_1,z_2) \doteq C^\perp D(z_1,z_2) (I-\mu B(z_1,z_2)) C^{\perp\top},  \\
      &\Phi(z_1,z_2) \doteq [ C^\perp D(z_1,z_2) (I- \mu B(z_1,z_2))C^\dagger, 0]
    \end{split}
    \end{equation*}
    and $N$ is defined in Eq.~\eqref{eq:N}. 
    Since $M$ depends on the pair $(z_1,z_2)$, establishing~\eqref{eq:contract} requires uniform bounds on the blocks of $M$, holding for any $(z_1,z_2)$. 
    Ass.~\ref{assumption_f}, guarantees that $ \|I - \mu B\|_2 \leq \max\big\{|1-\mu m_f|,\,|1-\mu L_f|\big\} \doteq c$. We have $c < 1$ for $\mu \in \left(0,\tfrac{2}{L_f}\right)$ and $m_f > 0$. Exploiting the properties of $D$ in Lemma~\ref{prop_D} and $\|C^\perp\|_2 = 1$, we have
    \begin{equation}\label{eq:bound_A_Phi}
        \|A(z_1,z_2)\|_2 \le c, \quad
        \|\Phi(z_1,z_2)\|_2 \le c\,\|C^\dagger\|_2 \doteq \varphi.
    \end{equation}     
   If we define $P_N \doteq Q^{-1} \succ 0$ and $\underline{\lambda}_N \doteq \lambda_{\min}(P_N) > 0$, condition~\eqref{eq:cond_N_stab} is equivalent to $\|N\|_{P_N} \le \sqrt{1-r_e} < 1$.   
   
   For $\theta > 0$ we define the weight $ P \doteq \begin{bmatrix} I_{n-m} & 0 \\ 0 & \theta^2 P_N \end{bmatrix}$, and we split $\Delta z = (z_1 - z_2)$ as $\Delta z = [\Delta \eta^\top, \Delta w^\top]^\top$, with $\Delta w \doteq [\Delta y^\top, \Delta v^\top]^\top$. 
    Using Young's inequality, which holds for any $\varepsilon>0$, and the fact that $\underline{\lambda}_N \| \cdot \|_2^2 \leq \| \cdot \|_{P_N}^2$, we obtain
    \begin{align*} 
    & \|M \Delta z\|_P^2 = \|A \Delta \eta + \Phi \Delta w\|_2^2 + \theta^2 \|N \Delta w\|_{P_N}^2\\
    &\le (1+\varepsilon ) c^2 \|\Delta \eta\|_2^2 + \left( 1+\frac{1}{\varepsilon} \right) \varphi^2 \|\Delta w\|_2^2 +\\
    & ~~~ + \theta^2(1-r_e)\|\Delta w\|_{P_N}^2 = r^2 \| \Delta z\|_P^2,
    \end{align*}
    where $r^2 = \max \left( \! (1+\varepsilon)c^2; \left( \!1+\frac{1}{\varepsilon} \!\right)\!\frac{\varphi^2}{\underline{\lambda}_N \theta^2} + (1-r_e) \! \right).$
    We have $r \in (0,1)$ if $\varepsilon \in \left(0, (1-c^2)/c^2\right)$ and $\theta^2 > \left(1+\varepsilon^{-1}\right){\varphi^2}/{r_e \underline{\lambda}_N}$.
    
    By the Banach fixed-point theorem, the contraction~\eqref{eq:contract} implies that~\eqref{eq:f_z} admits a unique fixed point $z^\star \in \R^{n+m}$ and that
    $\|z_k - z^\star\|_P \le r^k \|z_0 - z^\star\|_P$ for all $k \in \mathbb{N}$,
    i.e., $z^\star$ is globally exponentially stable.
    Finally, since the coordinate transformation is affine and invertible,
    there exists $\bar\kappa > 0$ such that
    \begin{equation*}\label{eq:dx_bounded_by_dz}
        \left\|\begin{bmatrix} x - x^\star \\ v - v^\star \end{bmatrix}\right\|_P
        \le \bar\kappa \, \|z - z^\star\|_P ,
    \end{equation*}
    and norm equivalence in finite dimension preserves the exponential decay.
    Thus the dynamics~\eqref{eq:FL-prox-CMO} admit a unique globally
    exponentially stable equilibrium $(x^\star,v^\star)$.

%
\subsection{Convergence analysis: proximal P\L{} condition}\label{CA2}
In this section, we analyze the convergence of~\eqref{eq:FL-prox-CMO} assuming that the cost function of problem \eqref{eq:opt_problem} satisfies the proximal P\L{} condition introduced in~\cite{karimi16}. 

\begin{definition}[Proximal P\L{} condition~\cite{karimi16}]\label{definition_PL}
    Let us consider the composite objective function $F(x)=f(x)+g(x)$. Assume $f$ is differentiable and $L_f$-smooth. Given $\mu>0$, we say that $F$ satisfies the proximal P\L{} condition if there exists $\alpha >0$ such that, for all $x$,
    \begin{equation}\label{eq: pl_ineq}
        \frac{1}{2} D_g(x, \mu) \geq \alpha (F(x) - F(x^*)),
    \end{equation}
    where $F(x^*)$ is the minimum value of $F$ over $\R^n$, and
    \begin{align*}
        D_g(x, \mu)
        = \frac{-2}{\mu} \min_{y \in \mathbb{R}^n} \Big\{
        & \langle \nabla f(x), y-x \rangle + \frac{1}{2 \mu}\|y - x\|_2^2 \\
        & + g(y) - g(x)
        \Big\}.
    \end{align*}
\end{definition}
%
%
For differentiable functions $f$, the proximal P\L{} inequality \eqref{eq: pl_ineq} reduces to the original formulation in~\cite{polyak63}, i.e., the condition
\begin{equation}
    \frac{1}{2}\|\nabla f(x)\|^2 \geq \alpha (f(x) - f(x^*)).
\end{equation}

Let $\mathcal{C} \doteq \{x \in \R^n : Cx+d=0\}$, $F^\star$ the optimal value of Problem~\eqref{eq:opt_problem}, and 
$\mathcal{P}_{\mathcal{C}}(w) \doteq C^{\perp\top}C^{\perp} w - C^\dagger d$ the orthogonal projection onto the affine set $\mathcal{C}$. We further denote by
\begin{equation}\label{eq:Fbar}
    \bar g \doteq g + \iota_{\mathcal{C}}, \qquad \bar F \doteq f + \bar g = F + \iota_{\mathcal{C}}
\end{equation}
the nonsmooth term and the objective of the constrained problem~\eqref{eq:opt_problem}; under Assumption~\ref{assumption_g}, $\bar g$ is proper, lower semi-continuous and convex, and $F^\star = \min_{x\in\R^n} \bar F(x)$.


\begin{assumption}\label{assumption_proj}
    The proximal operator of $g$ is compatible with the equality constraint, i.e., for all $w \in \R^n$,
    \begin{equation}\label{eq:prox_split}
        \mathcal{P}_{\mathcal{C}}\big(\prox_{\mu g}(w)\big) = \prox_{\mu \bar g}(w).
    \end{equation}
\end{assumption}
Assumption~\ref{assumption_proj} holds, e.g., when $g = \iota_{\mathcal{S}}$ with the constraint $x \in \mathcal{S}$ inactive on a neighbourhood of the optimal solution.
%
\begin{theorem}\label{th2}
    Let Assumptions~\ref{assumption_g}, \ref{assumption_C} and~\ref{assumption_proj} hold and let $f$ have an $L_f$-Lipschitz continuous gradient.
    Assume that $\bar F = f+g+\iota_{\mathcal{C}}$ satisfies the proximal P\L{} condition given in Definition~\ref{definition_PL} with $\mu = 1/L_f$ and $\alpha \in (0, L_f)$. 
    Then, the set of optimal solutions of~\eqref{eq:opt_problem} is globally exponentially stable for the dynamics~\eqref{eq:FL-prox-CMO}.
\end{theorem}
\textbf{Proof.}~
As shown in the proof of Theorem~\ref{th1}, the closed loop $z^+=F(z)$ is a cascade. The
$(y,v)$-subsystem is Schur by design and autonomous. Since $\Pi(x)$ is globally Lipschitz with constant $1+\mu L_f$, exponential stability of~\eqref{eq:FL-prox-CMO} follows from that of the zero dynamics
\begin{equation}\label{eq:zero_dynamics}
    \eta^+ = C^\perp \Pi(\xi), \quad \text{with} \quad
    \, \xi \doteq C^{\perp\top}\eta - C^\dagger d,
\end{equation}
obtained by setting $y=0$ in~\eqref{eq:f_z}, by standard ISS arguments for discrete-time cascades~\cite{jiang01}. 
Since $C^\perp C^\dagger = 0$, the variable $\xi$ evolves on the affine set $\mathcal{C}$ according to
\begin{equation}\label{eq:zero_dyn_xi}
    \xi^+ = 
    \mathcal{P}_{\mathcal{C}}(\Pi(\xi)) = \prox_{\mu \bar g}\big(\xi - \mu \nabla f(\xi)\big),
\end{equation}
where the last equality follows from Assumption~\ref{assumption_proj}.

Let us consider the candidate Lyapunov function $V(\xi) \doteq \bar F(\xi) - F^\star$. 
$V$ is a valid candidate Lyapunov function since $\eta \mapsto \xi = C^{\perp,\top} \eta - C^\dagger d$ is a bijection from $\R^{n-m}$ onto $\mathcal{C}$ and $V(\xi) > 0$ for all non-optimal $\xi$ as,
    \begin{equation}
        F^\star = \min_{x \in \mathcal C} F(x) = \min_\eta \bar F(C^{\perp,\top} \eta - C^\dagger d) \le \bar F(\xi).
    \end{equation}
    \noindent Since $f(\xi)$ is $L_f$-smooth and $\mu = 1/L_f$, we can apply \cite[Theorem~5]{karimi16} to $\bar F$ along~\eqref{eq:zero_dyn_xi}, yielding
    \begin{equation}
          \bar F(\xi^+) \leq \left(1- \mu \alpha \right) \bar F(\xi) + \mu \alpha F^\star,
    \end{equation}
    which, using the definition of $V(\xi)$, leads to
    \begin{equation}\label{eq:lyap_uncons_xi}
         V(\xi^+)\le \left(1- \mu\alpha\right) V(\xi).
    \end{equation}
    Eq.~\eqref{eq:lyap_uncons_xi} gives the linear decay of the objective, $\bar F(\xi_k) - F^\star \le (1-\mu\alpha)^k(\bar F(\xi_0)-F^\star)$. 
    Using $\alpha \in (0,L_f)$, $0 < 1 - \mu \alpha < 1$, and the result follows by Lyapunov argument.

\section{Numerical results}\label{sec:NR}
\subsection{Nonconvex cost function}
We consider a problem of the kind~\eqref{eq:opt_problem}, with nonconvex cost function $f(x) = x_1^4 + x_2^4 - x_1^2 - x_2^2$, $C = [1,1]$, $d=0.3$, and
where $g(x) = \iota_{P}(x)$ is the indicator function of the set $P = \{x \in \R^2: x_1 \leq -\frac{\sqrt{6}}{6}, x_2 \geq \frac{\sqrt{6}}{6}\}$. Fig.~\ref{fig:es1_plot} illustrates the setting. $f(x)$ has multiple local minima and a maximum in $x = 0$. It is easy to show that $\bar F$ is strongly convex on its domain, thus satisfying the proximal P\L{} inequality. Assumptions~\ref{assumption_g},~\ref{assumption_C}, and~\ref{assumption_proj} hold as well. Specifically,~\ref{assumption_proj} holds in a neighborhood of $x^\star$ according to the discussion below Assumption~\ref{assumption_proj}. Consequently, the conclusions of Theorem~\ref{th2} hold on such a neighborhood. 
%

In Fig.~\ref{fig:es1_plot}, we show the sequence of iterates generated by the proposed algorithm in Eq.~\eqref{eq:FL-prox-CMO}. We observe that the algorithm converges to the global feasible optimum, regardless of the initial condition. This confirms its effectiveness for a class of optimization problems with nonconvex $f(x)$, provided that $g(x)$ is suitably selected.
\begin{figure}
    \centering
    \includegraphics[width=\linewidth]{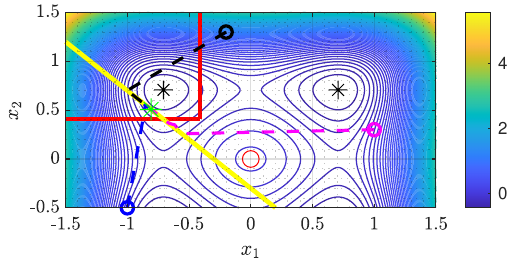}
    \caption{Level sets of $f(x)$ (see colorbar), maximum of $f(x)$ (${\color{red}\circ}$), local minima of $f(x)$ ($*$), contour of the set $P$  (\tikzref[color=red!95!white, line width=1.5pt]), equality constraint (\tikzref[color=yellow!75!white, line width=1.5pt]), global optimal solution (${\color{green}*}$). The trajectories generated by the algorithm~\eqref{eq:FL-prox-CMO} (\tikzref[color=black!90!white, line width=1.5pt, dash pattern=on 4pt off 1pt], \tikzref[color=magenta!80!white, line width=1.5pt, dash pattern=on 4pt off 1pt], \tikzref[color=blue!80!white, line width=1.5pt, dash pattern=on 4pt off 1pt]) with  initial conditions ({\color{black}$\circ$}, {\color{magenta}$\circ$}, {\color{blue}$\circ$}) achieve the global optimal solution.}
    \label{fig:es1_plot}
\end{figure}
\subsection{Equality-constrained Lasso}
In this section, we evaluate the performance of the proposed method on the equality-constrained Lasso problem, which is widely adopted in the context of sparse regression and compressed sensing~\cite{has15book} and is formalized as follows:
\begin{equation}
    \begin{aligned}
    \min_{x \in \mathbb{R}^n} \quad & \frac{1}{2} x^\top W x + \delta \| x\|_1 &    \text{s.t.} \quad & Cx + d = 0.
    \end{aligned}
\end{equation}
For this problem, $g(x) = \delta \norm{x}_1$ and the associated proximal operator admits a closed form representation known as the soft thresholding function, whose component-wise expression is $\sign (v_i) \max ( |v_i| - \delta,0)$.

We generate a symmetric positive definite $W \in \mathbb{R}^{n, n}$, and $C \in \mathbb{R}^{m, n}$, $d \in \mathbb{R}^{m}$ with independent and normally distributed entries. We fix $\delta=1$. With these choices Assumptions~\ref{assumption_g}, \ref{assumption_C}, and \ref{assumption_f} are satisfied, thus the problem is strongly convex.
We select $n = 100$ and $m = 50$, and perform $100$ random runs, each with newly generated data. 

We evaluate the average computational time, the number of iterations, and the percentage of zero components in the computed solutions. 
Table~\ref{table:confront_algs} shows the results and compares them with the solutions obtained by \texttt{CVX MATLAB}'s package using \texttt{SDPT3} solver~\cite{sdpt3} and with ADMM, which is an effective algorithm for constrained and composite optimization problems (see, e.g.,~\cite{boy10}). Compared to \texttt{SDPT3}, our method needs more iterations, but each iteration is cheaper, since \texttt{SDPT3} uses an interior-point method that requires factoring large matrices at every step. Compared to ADMM, our approach reduces both runtime and the number of iterations. 
\begin{table*}
        \centering
        \begin{tabular}{c|c c c}
         & {Total Runtime [\textrm{s}]} & {Iterations} & {Sparsity [\texttt{\%}]} \\
        \hline
        SDPT3      & 0.17 & 12 & 10.81 \\
        ADMM  &  0.14 & 11351 & 10.81  \\
        Ours~\eqref{eq:FL-prox-CMO} &   0.05 & 6508 &  10.81\\
        \end{tabular}
        \caption{Average metrics over 100 runs for equality-constrained Lasso. Sparsity denotes the ratio between null components and the dimension of the solution.}
        \label{table:confront_algs}
\end{table*}
\subsection{Training of a neural network}
We consider the nonlinear and nonconvex problem of training a single-hidden-layer feedforward neural network, i.e., estimating the parameters $w_1, w_2, b_1 \in \R^\nu$ and $b_2 \in \R$ of a function $\mathcal{N}: \R \rightarrow \R$ described by
\begin{equation}
    \mathcal{N}(u) = w_2^\top \,h( w_1 u + b_1) + b_2,
\end{equation}
where $\nu$ denotes the number of neurons and $h(\cdot)$ is the activation function. We assume that a set of measurements $\{u_i,y_i\}_{i=1}^N$, possibly perturbed by noise, is available for estimation. We train the network by solving the minimization problem $\min_{\theta} J(\theta)$ with loss defined by
\begin{equation}\label{eq:uncons_formul_loss}
    J(\theta) \doteq \left( \sum_{i=1}^N y_i - w_2^\top h( w_1 u_i + b_1) - b_2 \right)^2,
\end{equation}
where $\theta = [w_1^\top,b_1^\top,w_2^\top,b_2]^\top \in \R^{3\nu+1}$ collects all the parameters to be estimated.

Next, we show that for ReLU activation $h(z) = \max(0,z)$, the problem~\eqref{eq:uncons_formul_loss} can be recast in the nonsmooth composite constrained form~\eqref{eq:opt_problem}. Similar formulations have been considered in~\cite{Taylor_Burmeister_Xu_Singh_Patel_Goldstein_2016, Gao_Goldfarb_Curtis_2020}, where ADMM is employed for training. Here, we demonstrate that the proposed algorithm is a viable alternative in this context. {As discussed in~\cite{Taylor_Burmeister_Xu_Singh_Patel_Goldstein_2016}, the constrained optimization formulation exhibits advantages such as increased robustness to poor conditioning and gradient saturation.} To recast~\eqref{eq:uncons_formul_loss} in the form~\eqref{eq:opt_problem}, we introduce auxiliary variables $z_i,a_i \in \R^{\nu}$ defined for each $i \in \{1,\dots,N\}$ as
\begin{equation}
   z_i \doteq w_1u_i + b_1, \qquad a_i \doteq h(z_i).
\end{equation}
Collecting all unknown variables as $a = [a_1^\top,\dots,a_N^\top]^\top \in \R^{N\nu}$, $z = [z_1^\top,\dots,z_N^\top]^\top \in \R^{N\nu}$, and $x = [\theta^\top,z^\top,a^\top]^\top \in \R^{n_x}$, $n_x = 3\nu+1+2\nu N$, the Eq.~\eqref{eq:uncons_formul_loss} is recast to~\eqref{eq:opt_problem} with
\begin{subequations}
\begin{align}
   & f(x) \doteq \sum_{i=1}^N (y_i - w_2 a_i - b_2)^2, \quad d = 0, \\
   & C \doteq \begin{bmatrix}
        u \kron I_{\nu} & \mathbf{1}_N \kron I_{\nu} & 0_{N \nu, \nu +1} & -I_{N\nu} & 0_{N\nu}
    \end{bmatrix}.
\end{align}
\end{subequations}
The indicator function $g(x) = \iota_{P}(x)$ characterizes the set
\begin{equation}
    P \doteq \{x=[\theta^\top,a^\top,z^\top]^\top \in \R^{n_x}: a = h(z)\}.    
\end{equation}

The set $P$ is nonconvex, but its decoupled structure enables explicit computation. Specifically, $P$ can be partitioned as $P = I \times P_1 \times \dots \times P_{\nu N}$, where $I$ is the identity map for the unconstrained variables $\theta$, and each $P_i$ represents the scalar relation $a_{i,j} = h(z_{ij})$.
As a result, the projection reduces to solving a collection of independent scalar subproblems:
\begin{equation}
    z_{ij}^\star = \argmin{z \in \R}  (z-z_{ij})^2 + (h(z)-a_{ij})^2,
\end{equation}
which admit the following closed-form solution:
\begin{equation}
z_{ij}^\star =
\begin{cases} 
 z_{ij}, ~~~ \text{if}~z_{ij}\leq 0 ~\text{and}~a_{ij}\leq (1+\sqrt{2})\lvert z_{ij}\rvert \\
 \max \left(0,\frac{z_{ij} + a_{ij}}{2}\right) ~~~~ \text{otherwise}. 
\end{cases}
\label{eq:relu_projection}
\end{equation}

Fig.~\ref{fig:training} shows the evolution of the loss function during training computed on both training and validation datasets. We compare~\eqref{eq:FL-prox-CMO} with the Adam algorithm~\cite{KingmaB14}. The proposed method demonstrates superior performance, achieving lower loss values than Adam within a reduced number of iterations.
Additionally, we evaluate the final model's coefficient of determination on the test dataset, defined as
\begin{equation*}
R^2 = 1 - \frac{\sum_{i=1}^{n}(y_i - \hat{y}_i)^2}{\sum_{i=1}^{n}(y_i - \bar{y})^2},
\end{equation*}
where $y_i$ denotes the true value, $\hat{y}_i$ the predicted value, and $\bar{y}$ the mean of the true values, obtaining $R^2 = 0.99$ for the proposed algorithm \eqref{eq:FL-prox-CMO} and $R^2 = 0.98$ for Adam.

%
\begin{figure}
    \centering
    \includegraphics[width=\linewidth]{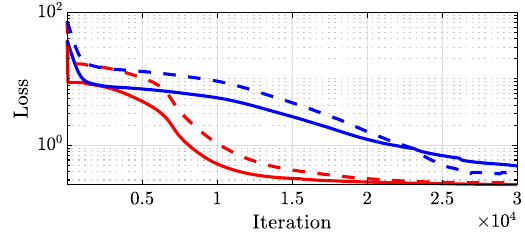}
    \caption{Evolution of the loss function during training:~\eqref{eq:FL-prox-CMO} (\tikzref[color=red!50!white, line width=1.5pt]) and Adam (\tikzref[color=blue!50!white, line width=1.5pt]). Evolution of the validation loss function:~\eqref{eq:FL-prox-CMO} (\tikzref[color=red!50!white, line width=1.5pt, dash pattern=on 4pt off 1pt]) and Adam (\tikzref[color=blue!50!white, line width=1.5pt, dash pattern=on 4pt off 1pt]). }
    \label{fig:training}
\end{figure}

\section{Conclusions}\label{sec:CON}
In this work, we present a discrete-time algorithm for linearly constrained nonsmooth optimization. 
We leverage the controlled multiplier optimization framework to define a dynamical system associated with the optimization problem and regulate it using feedback linearization, while handling nonsmooth terms through proximal operators.
We prove global exponential stability both in the strongly convex case and under the Polyak--\L{}ojasiewicz condition.
Finally, we validate our approach on three numerical experiments designed to capture different problem characteristics, thereby demonstrating the broad applicability of the method.
Future work will extend the proposed method to nonlinear equality constraints and apply it to relevant optimization problems arising in real-world engineering.

\balance
\bibliographystyle{ieeetr}
\bibliography{bibfile}

\end{document}